\documentclass[12pt]{article}

\usepackage[margin=1.25in]{geometry} % Adjust to 1.5in for wider margins, or 1in for wider text area

\usepackage{amsmath}
\usepackage{amssymb}
\usepackage{amsthm}
\usepackage{amsfonts}
\usepackage{mathrsfs}
\usepackage{graphicx}
\usepackage{xcolor}
\usepackage[round]{natbib}

\usepackage{authblk}

\newtheorem{theorem}{Theorem}
\newtheorem{proposition}{Proposition}
\newtheorem{lemma}{Lemma}

\newtheorem{definition}{Definition}

\newcommand{\RR}{\mathbb{R}}
\newcommand{\NN}{\mathbb{N}}

\DeclareMathOperator{\diam}{diam}

\title{On the geometry of weak convergence without total variation convergence}

\author{Nicola Bariletto \hspace{2cm} Stephen G. Walker}

\affil{The University of Texas at Austin}

\date{\today}

\begin{document}

\maketitle

\begin{abstract}
We study some geometric consequences of the discrepancy between weak and total variation convergence of probability measures. We consider a sequence of probability measures on $\RR^d$, admitting densities with respect to the Lebesgue measure, that converge weakly to a limiting measure but stay bounded away from it in total variation distance. We show that the sets on which the sequence passes from below to above the limiting density must grow unboundedly in perimeter, as measured by the $(d-1)$-dimensional Hausdorff measure. Moreover, this growth persists within a fixed compact set, so that it must reflect an increase in the geometric complexity of these sets rather than only an unbounded expansion in ambient space. We further provide a sufficient condition under which the number of connected components of the sets diverges, recovering a behavior that is closely reminiscent of the one-dimensional case, in which the number of oscillations of the sequence of densities around the limit grows without bound. We also show that this condition cannot be dispensed with in general, by means of an explicit sequence of measures in the plane whose passing sets remain connected in a single component at every stage while growing in length and complexity. Another sequence, built from cosine oscillations, illustrates the complementary behavior, in which the number of components diverges.
\end{abstract}

\section{Introduction}\label{sec:introduction}

The convergence of  probability measures \citep{billingsley1999} is a classical topic in probability theory and lies at the core of many of its applications. In this article, we study some geometric aspects of the relationship between two standard modes of convergence of probability measures on $\mathbb R^d$, that is, total variation and weak convergence, in the case where the involved probabilities are absolutely continuous with respect to the Lebesgue measure. In particular, we prove that weak convergence paired with a lack of convergence in total variation distance implies a diverging size (in terms of $(d-1)$-dimensional Hausdorff measure) of the subset of $\mathbb R^d$ on which the non-converging sequence of densities $f_j$ passes from below to above the density $g$ of the limit measure. We further characterize various aspects of this phenomenon and illustrate it with concrete two-dimensional examples. In particular, our work reveals that the recently established results for $d=1$, in terms of a growing number of oscillations of $f_j$ around $g$ \citep{bariletto2025identifiability}, translate to higher dimensions, but only through a substantial upgrade in both technical machinery and underlying insight: the meaningful higher-dimensional analogue of the number of oscillations is not their count, but the $(d-1)$-dimensional Hausdorff measure of the set on which $f_j$ passes from below to above $g$, with the two notions coinciding in general only in the one-dimensional case.

It is a well-known fact that total variation convergence implies weak convergence, and our focus in this work is on the  consequences of the discrepancy between these two modes of convergence. Besides this being a topic encompassing fundamental concepts in probability theory, our interest is motivated by Bayesian asymptotic statistics, and especially by a celebrated result known as \emph{Schwartz's consistency theorem} \citep{schwartz1965}. Loosely speaking, the theorem says that, under a mild prior support condition, the posterior distribution arising from a dominated likelihood model concentrates, as the size of the sampled dataset increases to infinity, within any weak neighborhood of the data-generating distribution. Often, however, one is interested in ascertaining posterior concentration within open sets belonging to a stronger topology, such as that induced by the total variation distance, which effectively measures the discrepancy between the densities (Radon-Nikodym derivatives) associated to any two probability measures. In cases where the two topologies agree, total variation contraction of the posterior is immediately deduced from contraction in weak neighborhoods via Schwartz's theorem. However, this does not hold in general, in particular in the very common scenario where the dominating measure is the Lebesgue measure, which is adopted to model continuous Euclidean data.

To strengthen the posterior contraction result of Schwartz to total variation neighborhoods, a number of contributions have proposed sufficient conditions in the form of regularity assumptions for sets of densities (known as sieves) on which the prior puts most of its mass \citep{barron1999,ghosal1999,walker2004squarerootsum}. Recent work has instead tackled the problem by highlighting the pathological consequences of weak posterior contraction in the absence of total variation contraction \citep{walker2005data,bariletto2025identifiability,bariletto2025necessary}. In particular, \cite{bariletto2025identifiability}, focusing on families of densities with respect to the Lebesgue measure on the real line, showed that weak convergence accompanied by total variation non-convergence implies the existence of a sequence of densities $f_j$ that oscillate with arbitrarily high frequency around the density $g$ associated with the limiting probability measure $G$. A classic example is given by the sequence $f_j(x)=(1+\cos(2\pi jx))1_{[0,1]}(x)$, visualized in Figure~\ref{fig:cosine_1d}, whose associated probability measures converge weakly to the uniform measure $G$ on $[0,1]$ but remain bounded away from it in total variation, due to the increasingly oscillatory behavior of $f_j$ around the limiting density $g(x)=1_{[0,1]}(x)$.

\begin{figure}
    \centering
    \includegraphics[width=0.9\linewidth]{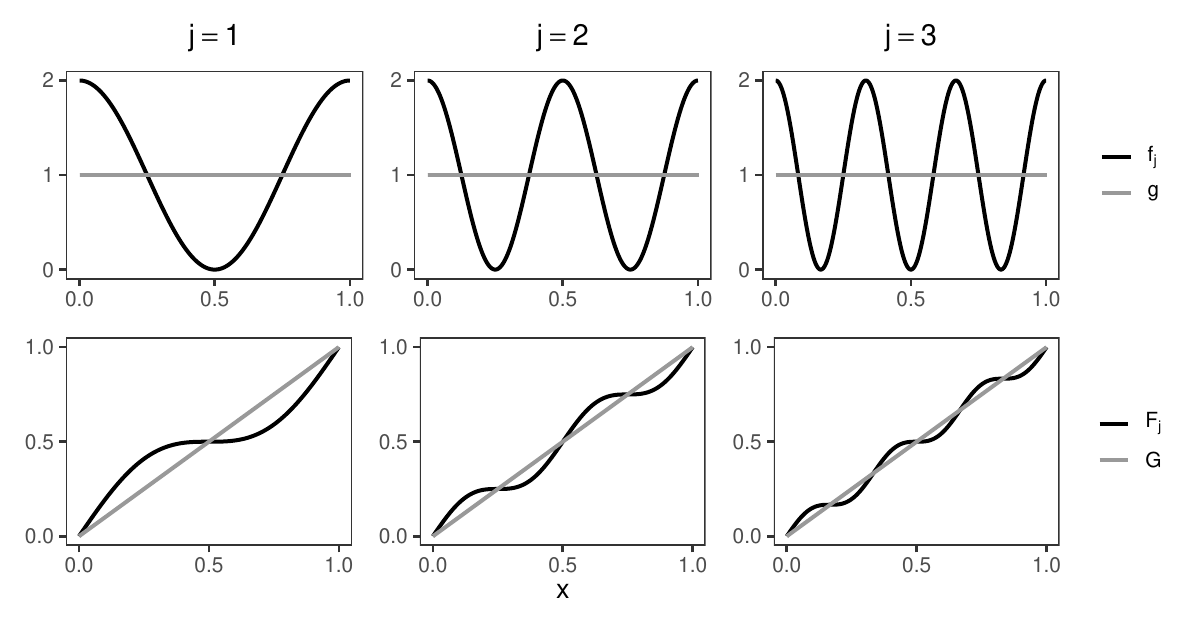}
    \caption{The one-dimensional sequence $f_j(x)=(1+\cos(2\pi jx))1_{[0,1]}(x)$ (black, top row) oscillates around $g(x)=1_{[0,1]}(x)$ (grey) with increasing frequency as $j$ grows. Although the densities do not converge, implying a lack of total variation convergence of the associated measures $F_j$, the latter converge weakly to $G$, as illustrated by the pointwise convergence of the cumulative distribution functions (bottom row) to that of $G$.}
    \label{fig:cosine_1d}
\end{figure}

In this work, our goal is to rigorously extend these insights to probability measures on $\mathbb R^d$, with $d\in\mathbb N$. This poses significant technical challenges, because while in $d=1$ it is possible to unambiguously describe the ``number of oscillations'' of $f_j$ around $g$, for instance as the number of connected components of the set $\{x\in\mathbb R : f_j(x)>g(x)\}$, and to show that its divergence to infinity satisfactorily describes the pathological consequences of weak without total variation convergence, in $d\geq 2$ these considerations no longer hold, and a conceptually more nuanced approach is required.

Our strategy for overcoming these difficulties is as follows. Consider again the $d=1$ cosine-based example of Figure~\ref{fig:cosine_1d}, where every density involved is continuous. In this case, the number of oscillations is meaningfully captured by the size, or cardinality, of the set $\{x\in[0,1]:f_j(x)=g(x)\}$, which equals $j+1$ for all $j\in\mathbb N$ in this particular instance. If instead each $f_j$ were a step function, the number of oscillations could be captured by the size of the boundary of the set $\{x\in\mathbb R : f_j(x)>g(x)\}$ (or, more formally, the boundary of its closure), where this boundary intuitively corresponds to the set of points $x$ at which $f_j$ ``jumps'' discontinuously from below to above $g$. In both cases, the key insight is that the number of oscillations of $f_j$ around $g$, whose growth properties we wish to study, is appropriately captured by the size, in the sense of the number of elements, of the set at which $f_j$ \emph{passes} from below to above $g$, whether this happens by $f_j$ intersecting $g$, by $f_j$ jumping discontinuously from below $g$ to above it, or by a combination of both. While a more formal definition of these sets will be given later, the intuition provided so far entitles us to refer to such sets generically as \emph{passing sets}.

This intuitive notion of passing set, where $f_j$ ``passes'' from below to above $g$, together with the idea of measuring its size as $j$ grows, constitutes the key conceptual step toward turning the $d=1$ result on increasing oscillations into a general statement in arbitrary dimensions. Once again to build visual intuition, consider $d=2$ and the densities plotted in Figure~\ref{fig:2d_oscillations}. Although not formally shown, the sequence of red densities is obtained as a mixture of Gaussian kernels that, as $j$ increases, weakly approximates the uniform measure $G$ on $[0,1]^2$; at the same time, due to a rapidly decreasing kernel variance as $j$ grows, the densities wiggle more and more markedly around the blue uniform density $g(x)=1_{[0,1]^2}$, thereby preventing total variation convergence. In this case too, since all densities are continuous, the increasingly oscillatory behavior of $f_j$ around $g$ may be intuitively captured by an increasing ``size'' of the passing set $\{x\in \mathbb R^2 : f_j(x)=g(x)\}$, which Figure~\ref{fig:2d_oscillations} visualizes by projecting it in orange onto a separate plane. Also in this case, more general notions of passing set can be considered for discontinuous or hybrid instances, without changing the substantive idea that such a set describes the collection of points at which $f_j$ passes from below to above $g$.

However, unlike in the $d=1$ case, it is not immediately clear which notion of ``size'' of the passing set one should use. In fact, the cardinality of the set itself becomes meaningless, as it is always infinite for these kinds of boundary surfaces. A related intuition, which is very close in spirit to the one-dimensional case, would be to count the number of connected components of the passing set (which, in Figure~\ref{fig:2d_oscillations}, grows as $1,2,5$ for $j=1,2,3$). Although well defined, this notion of size will be shown in our subsequent analysis to be inappropriate in general, as weak without total variation convergence may occur even in the presence of a passing set formed by a single connected component; we refer to Section~\ref{sec:example} for a concrete example of this phenomenon.

A visually valid measure of size in $\mathbb R^2$, and one that we will be able to show in general to diverge with $j$ as a consequence of weak without total variation convergence, is instead the \emph{arc length} of the passing set, which in Figure~\ref{fig:2d_oscillations} clearly increases with $j$. Of course, for $d>2$ the passing sets become higher-dimensional surfaces, so that arc length must be further generalized accordingly, and the \emph{$(d-1)$-dimensional Hausdorff measure} \citep{evans2015measure} will serve as a precise notion to that end.

\begin{figure}
    \centering
    \includegraphics[width=0.99\linewidth]{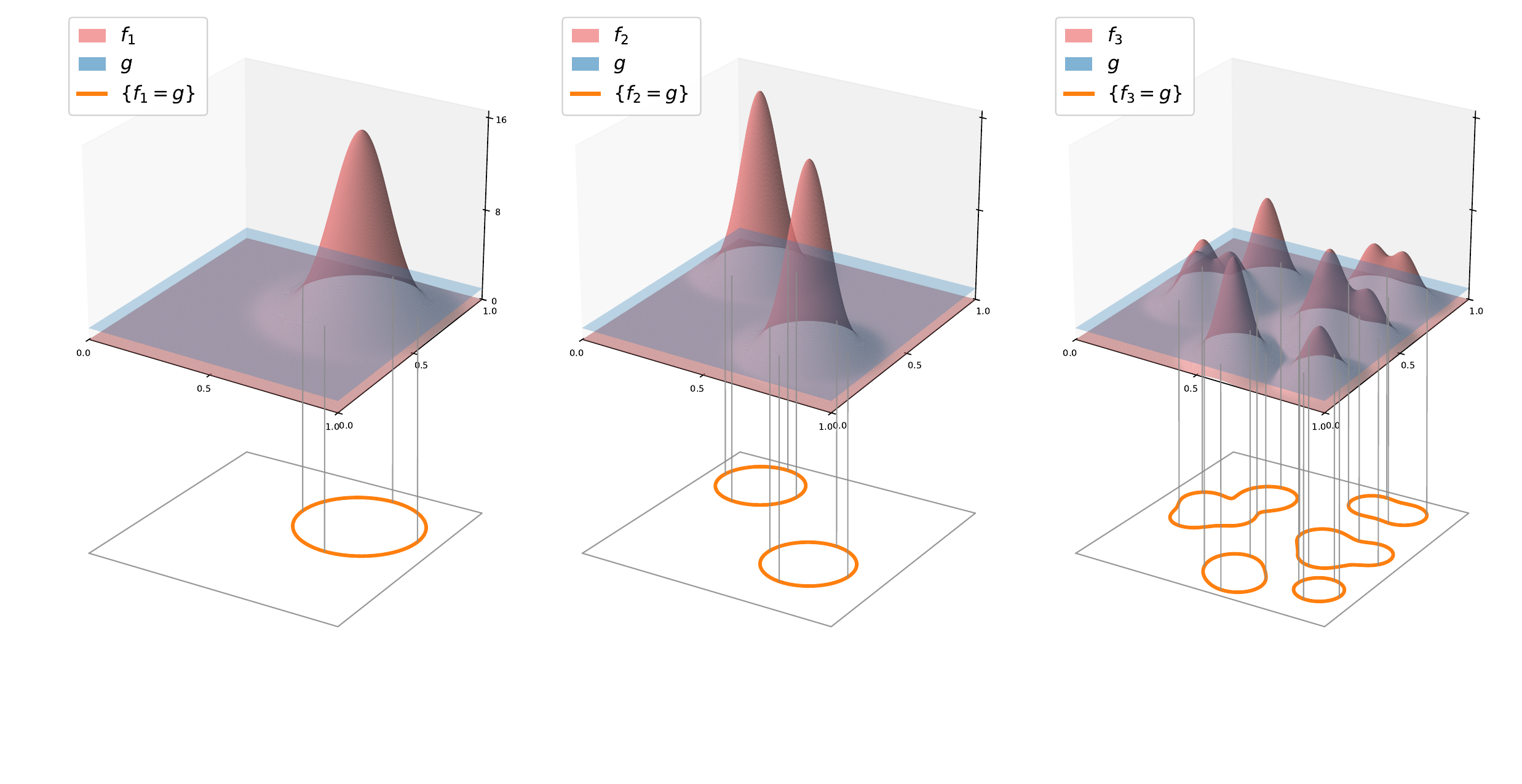}\vspace{-2em}
    \caption{Increasingly oscillatory densities $f_1,f_2,f_3$ (red) around the uniform $g$ on $[0,1]^2$ (blue). As oscillatory behavior builds up, the passing set $\{x\in\RR^2 : f_j(x)=g(x)\}$ (projected in orange onto the plane below each surface) gains arc length, the two-dimensional analogue of an increasing number of oscillations in one dimension.}
    \label{fig:2d_oscillations}
\end{figure}

Hence, using classical tools from geometric measure theory to extend the notion of oscillations to the multidimensional setting, our main contribution will be to show that if $(F_j)_{j\in\mathbb N}$ is a sequence of probability measures, all admitting a density $f_j$ with respect to the Lebesgue measure, that converges weakly to $G$ (with density $g$) but stays bounded away from it in total variation distance, then the sets at which the  densities $f_j$ pass from below to above $g$ must increase unboundedly in size, as measured by the $(d-1)$-dimensional Hausdorff measure (Theorem~\ref{thm:oscillations_high_dim_hausdorff}). In other words, just as \cite{bariletto2025identifiability} show that in one dimension the number of oscillations of $f_j$ around $g$ increases to infinity, we show that the size of an appropriate passing set diverges to infinity in arbitrary dimensions. Importantly, this notion of size can be formally related, when $d=1$, to the number of oscillations considered by \cite{bariletto2025identifiability}, of which it constitutes a meaningful higher-dimensional generalization (Proposition~\ref{prop:d1_recovery}). We further show that this divergence persists even after intersecting the passing sets with a fixed compact set (Theorem~\ref{thm:oscillations_high_dim_hausdorff_localized}), so that the uncovered growth in size of the passing sets represents a genuine increase in their geometric complexity rather than a balloon-like volume expansion in ambient space.

To compare with the one-dimensional result, we illustrate the phenomenon with two sequences of measures in $\mathbb R^2$. For the first sequence, the passing sets consist of a single connected component of increasing complexity. This defies the one-dimensional intuition that it must be the \emph{number} of sign changes of $f_j-g$ that diverges, and may be intuitively understood as follows: while the sequence of densities needs to wiggle ever more pronouncedly around the limit in order for total convergence to fail, in $d\geq 2$ there is more than one (indeed, an infinite number of) directions along which such densities may wiggle in the one-dimensional sense of an increasing number of oscillations; because this increase need not happen in every direction, it may be the case that the passing sets, while increasing in complexity, stay nonetheless connected in a single (or a bounded number of) components. This example shows that the basic one-dimensional picture fails in general in two dimensions or more, although we provide sufficient conditions under which it persists (Proposition~\ref{pro:components_diverge}). The second sequence, by contrast, exemplifies a multidimensional scenario in which the number of components of the passing sets grows without bound, demonstrating that this behavior, while not necessary, may occur in specific cases.

The rest of the article is organized as follows. Section~\ref{sec:notation} fixes the notation and collects the notions from probability theory and geometric measure theory that we require. Section~\ref{sec:from1d} reviews the one-dimensional result of \cite{bariletto2025identifiability}, discusses why its proof technique does not extend beyond the real line, and introduces the sets whose size we measure in the general case. Section~\ref{sec:main result} contains the main results, Section~\ref{sec:example} discusses examples, and Section~\ref{sec:conclusion} concludes.

\section{Preliminaries}\label{sec:notation}

\subsection{Basic notation}

Throughout, $d \in \NN$ denotes the ambient dimension and $\mathcal{B}(\RR^d)$ the Borel
$\sigma$-algebra on $\RR^d$. We write $\lambda_d$ for the $d$-dimensional Lebesgue measure on
$\RR^d$, $|\cdot|$ for the Euclidean norm, and $B(x,r)$ for the closed Euclidean ball of radius
$r>0$ centred at $x\in\mathbb R^d$. For a set $E \subseteq \RR^d$, we denote by $\overline{E}$ its closure, by
$E^c$ its complement, and by $\diam(E) := \sup_{x,y \in E}|x-y|$ its diameter. Probability
measures on $(\RR^d, \mathcal{B}(\RR^d))$ are denoted by capital letters $F, G, F_1, F_2, \dots$
and, when they are absolutely continuous with respect to $\lambda_d$, their densities are denoted
by the corresponding lower-case letters $f, g, f_1, f_2, \dots$. Integrals with respect to
$\lambda_d$ on $E\in\mathcal B(\mathbb R^d)$ are written $\int_E h(x)\,dx$ (with the convention that, when $E=\mathbb R^d$, it may be omitted from the notation). Finally, $1_E$ denotes the indicator function of the set $E$, and
for sequences $(a_j)_{j \in \NN}$, $(b_j)_{j \in \NN}$ we write $a_j = o(b_j)$ if
$a_j/b_j \to 0$ as $j \to \infty$.

% Standard background on measure theory and on the Lebesgue measure can be found in
% \citet{folland1999}; for the geometric-measure-theoretic notions used below we refer to
% \citet{evans2015measure}, \citet{maggi2012} and \citet{federer1969}.

\subsection{Probability metrics}

We work with two metrics on the space of probability measures on $\RR^d$. The first is the
\emph{total variation metric}
\begin{equation*}
d_{TV}(F,G) := \sup_{A \in \mathcal{B}(\RR^d)} |F(A) - G(A)|,
\end{equation*}
which, when $F$ and $G$ admit densities $f$ and $g$, satisfies the identity
$$d_{TV}(F,G) = \tfrac{1}{2}\int |f(x)-g(x)|\,dx$$ and is attained on the set
$\{x \in \RR^d : f(x) > g(x)\}$. The second is the \emph{L\'evy--Prokhorov metric}. For
$A \in \mathcal{B}(\RR^d)$ and $\delta > 0$, let
\begin{equation*}
A^{\delta} := \{x \in \RR^d : d_A(x) < \delta\}, \qquad
d_A(x) := \inf_{y \in A}|x-y|,
\end{equation*}
denote the open $\delta$-enlargement of $A$ and the Euclidean distance function to $A$,
respectively. The L\'evy--Prokhorov distance between $F$ and $G$ then is
\begin{align*}
d_w(F,G) := \inf\Bigl\{\delta > 0  :\; & F(A) \leq G(A^{\delta}) + \delta \text{ and }\\
&
G(A) \leq F(A^{\delta}) + \delta \text{ for all } A \in \mathcal{B}(\RR^d)\Bigr\}.
\end{align*}
On a separable metric space, $d_w$ metrizes weak convergence \citep{billingsley1999} and $d_{TV}$ dominates $d_w$ \citep{gibbs2002choosing}. Moreover, no general reverse inequality exists, providing a simple proof of the fact that total variation convergence is stronger in general than weak convergence; the discrepancy between these two notions of convergence is precisely the focus of this work.

\subsection{Hausdorff measure}

For any $E \in\mathcal B(\RR^d)$ and $r\geq 0$, the \emph{$r$-dimensional Hausdorff measure} of $E$ is defined as
\begin{equation*}
\mathcal{H}^{r}(E) := \lim_{\delta \to 0}\; \inf \left\{ \sum_{i=1}^\infty (\diam(U_i))^{r} : E \subseteq \bigcup_{i=1}^\infty U_i,\, \diam(U_i) < \delta \right\},
\end{equation*}
where the infimum is taken over all countable covers $(U_i)_{i\in\NN}$ of $E$ comprising sets $U_i$ all of
diameter less than $\delta>0$; see, for instance, Chapter 2 of \citet{evans2015measure}.

The case $r=d-1$ will be particularly relevant for our analysis. In fact, the sets we shall measure by means of $\mathcal H^{d-1}$ are formally subsets of $\RR^d$ but
correspond to $(d-1)$-dimensional boundary surfaces. In this scenario, applying the $d$-dimensional Lebesgue measure $\lambda_d$ to such
boundaries would trivially yield zero, and at the same time, one cannot use the $(d-1)$-dimensional
Lebesgue measure $\lambda_{d-1}$, as it is only properly defined for subsets of flat spaces such as
$\RR^{d-1}$ itself. The Hausdorff measure $\mathcal{H}^{d-1}$ resolves this by rigorously capturing the ``$(d-1)$-dimensional volume'' (e.g., arc length in two dimensions, surface area in three dimensions, etc.) of
potentially curved sets embedded in higher-dimensional space. In fact, when multiplied by
the geometric constant
$c_{r} := \pi^{r/2}[2^{r}\Gamma(r/2 + 1)]^{-1}$
evaluated at $r=d-1$, the measure $\mathcal H^{d-1}$ on $\mathbb R^{d-1}$
coincides with the Lebesgue measure on that same space:
\begin{equation*}
\lambda_{d-1}(E)=c_{d-1}\mathcal{H}^{d-1}(E) \quad \text{for all } E \subset \RR^{d-1};
\end{equation*}
see again Chapter 2 of \citet{evans2015measure}. Because of this direct correspondence,
$\mathcal{H}^{d-1}$ serves as a natural analogue of the Lebesgue measure for the
sets we wish to measure. Relatedly, $\mathcal H^{0}$ is easily seen to reduce to the
counting measure, so that measuring a subset of $\RR$ in the sense of $\mathcal H^0$ amounts to counting its
elements; this provides, as will highlighted by Proposition~\ref{prop:d1_recovery}, a precise connection to the one-dimensional result on the diverging number of oscillations by \cite{bariletto2025identifiability}.

\subsection{Coarea formula}

Finally, a key step in our proofs relies on the \emph{coarea formula} for Lipschitz continuous functions
\citep[see Chapter 3 of][]{evans2015measure}, which we state here for future reference: if
$u : \RR^d \to \RR$ is Lipschitz continuous and $E \in \mathcal B(\RR^d)$, then
\begin{equation}\label{eq:coarea}
\int_E |\nabla u(x)| \, dx = c_{d-1}\int_{\RR} \mathcal{H}^{d-1}\bigl(E \cap u^{-1}(t)\bigr)\,dt,
\end{equation}
where $\nabla u$ denotes the gradient of the function $u$ and $u^{-1}(t):=\{x\in\mathbb R^d : u(x)=t\}$ for all $t\in\mathbb R$.

The easiest way to visualize the content of the coarea formula is to focus
on the case $d=1$, so that $u:\RR\to\RR$, such as the function depicted in
Figure~\ref{fig:coarea}. Because $c_{d-1}=c_0=1$ and $\mathcal H^{d-1}=\mathcal H^0$
counts the elements of a set $E\in\mathcal B(\RR)$, the coarea formula in this simple
case states that integrating the size of the infinitesimal variation
$|\nabla u(x)|\equiv|\mathrm du(x)/\mathrm dx|$ over $x\in E$ is equivalent to
integrating the cardinality of the level sets $\{x\in\RR:u(x)=t\}$ over $t\in u(E)$.
Intuitively, a region where $u$ varies steeply contributes a large amount to the
left-hand side and is traversed by many level sets, hence adding to their cardinalities and therefore to the right-hand side of the formula. For $d>1$, the content of the formula is unchanged, up to the
dimension-dependent scaling factor $c_{d-1}$ and the reinterpretation of
$\mathcal H^{d-1}$ as the appropriate notion of $(d-1)$-dimensional size of the level
sets $u^{-1}(t)$.

\begin{figure}[t]
    \centering
    \includegraphics[width=0.8\linewidth]{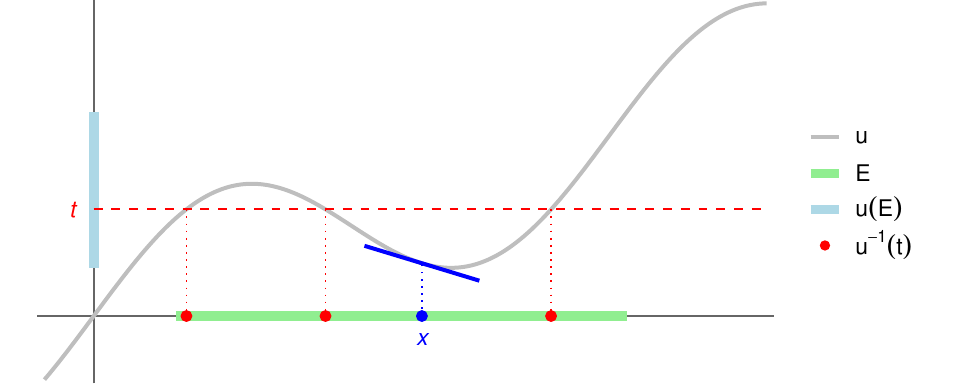}
    \caption{The coarea formula in one dimension. The graph of a Lipschitz continuous function
    $u:\RR\to\RR$ is shown in grey, with the set $E$ (green) marked on the
    horizontal axis and its image $u(E)$ (light blue) on the vertical axis. The tangent
    at a point $x$ (dark blue) has absolute slope $|\nabla u(x)|$. The horizontal line at level
    $t$ meets the graph of $u$ at three points, giving rise to the level set $u^{-1}(t)$
    marked on the domain axis (red). The coarea formula states that integrating $|\nabla u(x)|$ over $x\in E$ is equivalent to integrating the cardinality of the set $u^{-1}(t)$ over $t\in u(E)$.}
    \label{fig:coarea}
\end{figure}

\section{From one to arbitrary dimensions}\label{sec:from1d}

The study of conditions under which weak convergence of probability measures implies convergence
in total variation has a long history. For instance, \cite{hettmansperger1974note} showed that if
a sequence of distributions on $\RR$ converges weakly and the associated densities are unimodal,
then the densities converge to the density of the limiting distribution. \cite{dellacherie1978probabilities} later required $\liminf_{j\to\infty}f_j(x)\geq g(x)$ almost everywhere, while \cite{visintin1984strong} gave convexity conditions at the level of
the densities under which weak convergence implies total variation convergence. More recently,
\cite{walker2023comparing} relaxed unimodality to a finite bound on the number of modes along the
sequence, while \cite{mynbaev2026equivalence} further weakened Walker's conditions to only require total variation equicontinuity of the sequence of densities. It should be noted that these results only apply to sequences of measures on the real line, where modes
and related concepts admit an elementary description that is lacking in higher dimensions without imposing strong regularity conditions on the densities under consideration.

A related strand of work has shown that sufficient regularity of the sequence of densities yields
inverse bounds involving the Wasserstein distances \citep{Villani2009} rather than the
Lévy--Prokhorov metric, the former also metrising weak convergence provided additional moment
conditions hold. In particular, while the total variation distance cannot be upper-bounded by the
Wasserstein metrics in general, \cite{chae2020wasserstein} showed that for sufficiently smooth
densities it is in fact bounded by a power of the Wasserstein distance, with the bound depending on the
Sobolev norms of the densities. \cite{chae2024wasserstein} extended this to $L^p$-distances
between multivariate densities whose smoothness is measured in Besov norms. These results may be
intuitively linked to the oscillatory behavior of the sequence of densities, as the smoothness
requirements can be seen to serve the purpose of limiting precisely this kind of behavior and
ensuring a stronger form of convergence as a consequence.

\subsection{The one-dimensional case: oscillations}

A recent contribution to the topic is found in \cite{bariletto2025identifiability}. In that work, it was shown that, under openness of the set
$\{x\in\RR : f_j(x)>g(x)\}$, $F_j\to G$ weakly as $j\to\infty$, in conjunction with $d_{TV}(F_j,G)\geq \varepsilon$
for all $j\in\NN$, implies that the number of intervals making up that open set must
diverge to infinity.\footnote{We briefly note that the condition $d_{TV}(F_j,G)\geq \varepsilon$
for all $j\in\NN$ and some $\varepsilon>0$, which we also adopt in our analysis, is stronger than non-convergence of $F_j$ to $G$ in total variation. However, to ease notation, we work under this more restrictive scenario with the understanding that, in the general case of total variation non-convergence, our results apply along a subsequence.} The number of such intervals intuitively coincides with the number of
oscillations of $f_j$ around $g$, so that the sequence of densities $f_j$ must oscillate
arbitrarily often around the target $g$ in order for weak convergence (i.e., convergence of
integrals of appropriate test functions) to happen detached from total variation convergence. A classic example of this behavior, which \cite{bariletto2025identifiability} revisited in the context of Bayesian asymptotics, is given by the sequence of densities $f_j(x)=(1+\cos(2\pi jx))1_{[0,1]}(x)$, which oscillates with diverging frequency as $j\to\infty$, precluding total variation convergence, while the associated sequence of probability measures converges weakly to the uniform measure on $[0,1]$; see again Figure~\ref{fig:cosine_1d} for a visualization.

The proof techniques of \cite{bariletto2025identifiability} rely on the unique structure of
the real line, in particular on the fact that any open set in $\RR$, including
$\{x\in\RR : f_j(x)>g(x)\}$, can be written as the countable union of disjoint open intervals, and each interval is intuitively identified with an oscillation of $f_j$ around $g$. No analogous
decomposition is available in $\RR^d$ for $d\geq 2$, and, more fundamentally, the \emph{number} of oscillations that
the one-dimensional theorem characterizes is no longer a general enough measure of irregular or oscillatory behavior. Our upcoming results will clarify this point, and the examples of Section~\ref{sec:example}
will illustrate it with concrete sequences of densities.

\subsection{The multi-dimensional case: passing sets}

To carry out our analysis on more general Euclidean spaces, we propose to reason as follows. Still
focusing on probability measures on the real line, the number of of $f_j$ around $g$ can be
equivalently captured by the cardinality of the boundary of
$\{x\in\RR : f_j(x)>g(x)\}$, such boundary representing the set of points at which ``$f_j$ passes
$g$.'' Moreover, cardinality may be interpreted as the ``perimeter'' of such a
boundary set, which in regular cases will be $0$-dimensional and formed by a finite union of points. This view
can be translated to higher dimensions $d\in\NN$ by measuring the ``$(d-1)$-dimensional
perimeter'' of the boundary of $\{x\in\RR^d : f_j(x)>g(x)\}$ or of some closely related set, such as its closure or some appropriate enlargement thereof.

Figure~\ref{fig:2d_oscillations}, already discussed in the Introduction, presents a visual illustration with a two-dimensional case in which $f_j\neq g$ almost everywhere, where such a perimeter corresponds
to the arc length of the set $\{x\in\mathbb R^2 : f_j(x)=g(x)\}$. Importantly, while in Figure~\ref{fig:2d_oscillations} the boundary does fragment into an increasing number of disjoint components, one of the illustrative sequences in Section~\ref{sec:example} will clarify that this need not be the case when $d\geq 2$, and that the more meaningful measure of oscillatory behavior increasing when $F_j\to G$ weakly but not in total variation is precisely the perimeter of appropriately defined boundary sets, analogous the one pictured in Figure~\ref{fig:2d_oscillations}.

Before turning to a formal analysis, two further points need care. The first concerns the kind of probability measures on which our analysis will focus, for which a minimal regularity condition, embodied in the next definition, is required.

\begin{definition}\label{def:tv_regular}
A pair of probability measures $(F, G)$ on $\RR^d$ is called \emph{TV-regular} if there
exists a closed set $A \in \mathcal B(\RR^d)$ that achieves the total variation distance between
them, meaning $d_{TV}(F, G) = F(A) - G(A)$. When $F$ and $G$ admit densities $f$ and $g$, we write
$A(f,g)$ for one such set (which is understood to be arbitrarily chosen unless otherwise specified).
\end{definition}

While TV-regularity may at first appear arbitrary or restrictive, note that for measures
$F$ and $G$ with densities $f$ and $g$ the supremum defining $d_{TV}(F,G)$ is always attained on the set
$C(f,g):=\{x \in \RR^d : f(x) > g(x)\}$, so that the condition only requires this set to be replaceable by
a closed one. That is the case, for instance, whenever the boundary of
$\overline{C(f,g)}$ is Lebesgue-null, since $\overline{C(f,g)}$ then differs from $C(f,g)$
by a null set and attains the supremum as well. It is also the case whenever the
difference $h := f-g$ is continuous almost everywhere; that is, for TV-regularity to fail, $f-g$ needs to be discontinuous on a set of positive Lebesgue measure, which implies a high degree of pathology. To see this fact, let $W:=\{x\in\mathbb R^d : h(x)<0\}^\circ$, so that $W^c=\overline{\{x\in\mathbb R^d : h(x)\geq0\}}$. If
$h$ is continuous at a point $x$ with $h(x)<0$, then $h$ is strictly negative on an whole neighborhood of
$x$, so that $x\in W$; then $\{x\in\mathbb R^d : h(x)<0\}\cap W^c$ is contained in the
set of discontinuity points of $h$ and is consequently $\lambda_d$-null. Therefore
\begin{align*}
F(W^c)-G(W^c) & =\int_{W^c} h(x)\,dx\\
&=\int_{\{h>0\}} h(x)\,dx+ \int_{\{h=0\}} h(x)\,dx + \int_{\{h<0\}\cap W^c} h(x)\,dx
\\
&=\int_{\{h>0\}} h(x)\,dx \\
& = d_{TV}(F,G),
\end{align*}
showing that the closed set $W^c$ achieves the supremum in the total variation distance definition. Hence, in particular, TV-regularity
holds for continuous and piecewise-continuous densities, and more broadly for the pairs of
densities typically encountered in applications. Accordingly, while we retain the notion of
TV-regularity for the sake of formal generality, the set $A(f,g)$ will be interpreted almost without loss of generality as the region on which $f$ exceeds
$g$, that being the mechanism by which $d_{TV}(F,G)$ is attained on $A(f,g)$.

The second point concerns the boundary set that we aim to measure, which we replace by a level set of the
distance function to $A(f,g)$ at a strictly positive value, termed a \emph{passing set}.

\begin{definition}\label{def:passing_set}
Let $f$ and $g$ be probability densities on $\RR^d$ such that the associated pair of
distributions is TV-regular, and let $t>0$. The \emph{$(f,g,t)$-passing set} is defined as
\begin{equation*}
P_t(f,g) := \{x\in\RR^d : d_{A(f,g)}(x) = t\}.
\end{equation*}
\end{definition}

That is, given the running interpretation of the set $A(f,g)$, the passing set $P_t(f,g)$ lies at distance exactly $t$ from the set on which $f$ exceeds $g$, and
therefore, for small $t$, it tightly approximates the set at which $f$ passes, or goes from lying below to lying above, $g$.

%describes the geometry of the latter at resolution $t$: features separated by a distance
%smaller than $2t$ are not resolved. In the results below, the resolution $t$ is taken to vanish
%along the sequence, so that the geometry is inspected at increasingly fine scales, and no
%assumption beyond closedness of $A(f,g)$ is needed.

As already mentioned, our goal is to generalize the one-dimensional analysis linking the discrepancy between weak and total variation convergence of a sequence $F_j$ to $G$ with a diverging number of oscillations of $f_j$ around $g$. To that end, the next result formalizes the connection between the number of such oscillations and the measure $\mathcal H^0$ of an appropriate passing set: the two quantities bound each other up to additive or multiplicative constants, so that either diverges if and only if the other does.

\begin{proposition}\label{prop:d1_recovery}
Let $d=1$, $A\subset\RR$ non-empty and closed, $t>0$, and let
$N_t\in\NN\cup\{\infty\}$ denote the number of connected components of the open enlargement $A^t$. Then
\begin{equation*}
N_t - 1 \leq \mathcal H^{0}(\{x\in\RR : d_A(x)=t\}) \leq 2 N_t,
\end{equation*}
with the convention that both sides are infinite when $N_t=\infty$.
\end{proposition}

\begin{proof}
Write $L := \{x \in \RR : d_A(x) = t\}$ and recall that $\mathcal H^{0}$ coincides with the counting measure, so that the claim concerns the number of elements of $L$. Since $d_A$ is continuous, $A^{t}$ is open, so that the disjoint countable union of its connected components, each of which is a non-empty open interval, is well defined.% We show that $L$ is precisely the set of finite endpoints of these components, and then count.

Let $a$ be a finite endpoint of a component $U$ of $A^t$. Then $a\notin A^t$, so that $d_A(a)\geq t$, while approaching $a$ from within $U$ and using the continuity of $d_A$ gives $d_A(a)\leq t$. Hence $a\in L$. Conversely, let $x\in L$. Because $A$ is closed, the infimum defining $d_A(x)$ is attained at some $y\in A$ with $|x-y|=t$, and we may assume $y>x$, the other case being symmetric. Writing $x_s := x+s(y-x)$ for $s\in(0,1]$, we have
$
d_A(x_s) \leq |x_s - y| = (1-s)\,t < t ,
$
so that $(x,y]\subseteq A^{t}$. Being connected, $(x,y]$ lies in a single component $U=(a,b)$ of $A^t$, whence $a\leq x<y\leq b$. Since $d_A(x)=t$, the point $x$ does not belong to $A^t$ and therefore not to $(a,b)$; as $x<b$, this forces $a=x$, so that $x$ is the left endpoint of $U$. In the symmetric case $y<x$ one finds that $x$ is the right endpoint of the component containing $[y,x)$.

We are left to compare cardinalities. Let $M$ denote the number of pairs $(U,x)$ with $U$ a component of $A^t$ and $x$ one of its finite endpoints. Every component, being an open interval, has at most two finite endpoints, so $M\leq 2N_t$; and every $x\in L$ occurs in at most two such pairs, once as a right endpoint and once as a left endpoint, so $M\leq 2\,\mathcal H^0(L)$. By the two paragraphs above, every finite endpoint belongs to $L$ and every element of $L$ is a finite endpoint, so the first bound also reads $\mathcal H^0(L)\leq M\leq 2N_t$, which is the upper bound in the statement. For the lower bound, note that a component fails to contribute two finite endpoints only when it is unbounded, and $A^t$ admits at most one component unbounded from below and at most one unbounded from above, each such component losing a single endpoint. Hence $M\geq 2N_t-2$, and combining this with $M\leq 2\,\mathcal H^0(L)$ yields $N_t-1\leq \mathcal H^0(L)$.
\end{proof}

Consequently, in dimension one, divergence of $\mathcal H^{0}(P_{t_j}(f_j,g))$ along a vanishing sequence $t_j$, which is the form that our upcoming results take, is equivalent to divergence of the number of connected components of $A(f_j,g)^{t_j}$, that is, of the number of intervals composing the set on which $f_j$ exceeds $g$ (modulo a vanishing enlargement). In this sense, when $d=1$, the results of Section~\ref{sec:main result} return a version of the statement of \cite{bariletto2025identifiability} about oscillatory densities.

\section{Main results}\label{sec:main result}

We are now in a position to begin our formal analysis. The following theorem constitutes the main result of the article.

\begin{theorem}\label{thm:oscillations_high_dim_hausdorff}
Let $g, f_1, f_2, \dots$ be probability densities on $\RR^d$ with corresponding probability
distributions $G, F_1, F_2, \dots$. For each $j \in \NN$, assume that the pair $(F_j, G)$ is
TV-regular, and let $A_j := A(f_j,g)$. Suppose that, for some $\varepsilon>0$,
\begin{enumerate}
    \item[(i)] $d_{TV}(F_j, G) \geq \varepsilon$ for all $j \in \NN$,
    \item[(ii)] $d_w(F_j, G) \to 0$ as $j \to \infty$.
\end{enumerate}
Then, for any sequence $\delta_j > d_w(F_j, G)$ such that $\lim_{j \to \infty} \delta_j = 0$,
there exists  $t_j \in (0, \delta_j)$ such that
\begin{equation*}
\lim_{j \to \infty} \mathcal{H}^{d-1}\bigl(P_{t_j}(f_j,g)\bigr) = \infty.
\end{equation*}
\end{theorem}

The proof combines an application of the coarea formula \eqref{eq:coarea} with the following property of Euclidean
distance functions.

\begin{lemma}\label{lem:eikonal}
Let $A \subset \RR^d$ be a non-empty \emph{closed} set, and let $d_A: \RR^d \to [0,\infty)$ be
the Euclidean distance function to $A$, defined by
\begin{equation*}
d_A(x) := \inf_{y \in A} |x - y|.
\end{equation*}
Then $d_A$ is $1$-Lipschitz continuous on $\RR^d$, differentiable $\lambda_d$-almost everywhere, and satisfies
\begin{equation*}
|\nabla d_A(x)| = 1 \quad \text{for almost every } x \in \RR^d \setminus A.
\end{equation*}
\end{lemma}

\begin{proof}
First, we establish the $1$-Lipschitz property via the triangle inequality. For any
$x, z \in \RR^d$ and $y \in A$, we have $|x - y| \leq |x - z| + |z - y|$. Taking the infimum over
all $y \in A$ on both sides yields $d_A(x) \leq |x - z| + d_A(z)$, which implies
$d_A(x) - d_A(z) \leq |x - z|$. Reversing the roles of $x$ and $z$ gives
$|d_A(x) - d_A(z)| \leq |x - z|$.

Because $d_A$ is $1$-Lipschitz continuous on the open set $\RR^d \setminus A$, Rademacher's
Theorem guarantees that $d_A$ is differentiable almost everywhere in $\RR^d \setminus A$. Let
$x \in \RR^d \setminus A$ be a point where $\nabla d_A(x)$ exists.

\emph{Upper bound: $|\nabla d_A(x)| \leq 1$.}
Let $v \in \RR^d$ be an arbitrary unit vector ($|v| = 1$). By the definition of the directional
derivative operator $D_v$ and the Lipschitz condition, we have
\begin{equation*}
|D_v d_A(x)| = \left| \lim_{t \to 0} \frac{d_A(x + tv) - d_A(x)}{t} \right| \leq \lim_{t \to 0} \frac{|(x+tv) - x|}{|t|} = 1.
\end{equation*}
Since $D_v d_A(x) = \nabla d_A(x) \cdot v$, choosing $v = \frac{\nabla d_A(x)}{|\nabla d_A(x)|}$
(assuming $\nabla d_A(x) \neq 0$) yields
\begin{equation*}
|\nabla d_A(x)| = \nabla d_A(x) \cdot \frac{\nabla d_A(x)}{|\nabla d_A(x)|} \leq 1.
\end{equation*}

\emph{Lower bound: $|\nabla d_A(x)| \geq 1$.} Take $x \in \mathbb R^d \setminus A$ and let $r := d_A(x) > 0$. Choose a radius $R > r$ and define $K := A \cap B(x, R)$. The set $K$ is closed and bounded, hence compact. By the Extreme Value Theorem, the
continuous mapping $z \mapsto |x - z|$ attains its minimum on $K$ at some point $y \in K \subseteq A$.
Because any point $w \in A \setminus K$ satisfies $|x - w| > R > r$, no point outside $K$
can be the closest point. Thus, $y$ is a global minimizer over the entire set $A$, and
$d_A(x) = |x - y|$.

Consider a point $x_t$ on the straight line segment from $x$ to $y$, parametrized
by $t \in (0, |x-y|)$ with the unit direction vector $v = \frac{x - y}{|x - y|}$:
\begin{equation*}
x_t = x - tv.
\end{equation*}
By construction, the Euclidean distance from $x_t$ to $y$ is exactly $|x - y| - t$. Since
$y \in A$, the distance from $x_t$ to the set $A$ is bounded as
\begin{equation*}
d_A(x_t) \leq |x_t - y| = |x - y| - t = d_A(x) - t.
\end{equation*}
Rearranging terms yields
\begin{equation*}
\frac{d_A(x) - d_A(x - tv)}{t} \geq 1.
\end{equation*}
Taking the one-sided limit as $t \to 0^+$ on both sides, the left-hand side converges to the
directional derivative $D_v d_A(x)$ at $x$ along direction $v$:
\begin{equation*}
D_v d_A(x) = \nabla d_A(x) \cdot v \geq 1.
\end{equation*}
Applying the Cauchy--Schwarz inequality, the inner product is bounded as
$1 \leq \nabla d_A(x) \cdot v \leq |\nabla d_A(x)||v|$, and since $|v| = 1$, it follows that
$|\nabla d_A(x)| \geq 1$.

Finally, combining both bounds, we conclude that $|\nabla d_A(x)| = 1$ at every point of differentiability
in $\RR^d \setminus A$; this, in conjunction with the $\lambda_d$-almost everywhere differentiability of $d_A$ on that same set, finishes the proof.
\end{proof}

We can now prove the main result.

\begin{proof}[Proof of Theorem~\ref{thm:oscillations_high_dim_hausdorff}]
By assumption, $A_j$ is a closed set that achieves the total variation distance, so that
\begin{equation*}
d_{TV}(F_j, G) = F_j(A_j) - G(A_j) \geq \varepsilon.
\end{equation*}
By the definition of the L\'evy--Prokhorov distance, since $\delta_j > d_w(F_j, G)$, the
inequality $F_j(A) \leq G(A^{\delta_j}) + \delta_j$ holds for every Borel set $A$. 
Applying this to the set $A = A_j$ gives
\begin{equation*}
F_j(A_j) \leq G(A_j^{\delta_j}) + \delta_j.
\end{equation*}
Combining these inequalities, we obtain
\begin{equation*}
\varepsilon + G(A_j) \leq F_j(A_j) \leq G(A_j^{\delta_j}) + \delta_j \implies G(A_j^{\delta_j} \setminus A_j) \geq \varepsilon - \delta_j.
\end{equation*}
Notice that $A_j^{\delta_j} \setminus A_j = \{x \in \RR^d : 0 < d_{A_j}(x) < \delta_j\}=:S_j$
because $A_j$ is closed. Moreover, let $E_M:=\{x\in\RR^d : g(x)>M\}$ for all $M>0$, so that
$G(E_{M})=\int_{\RR^d}1_{E_M}(x) g(x)\,dx = 1 - \int_{\RR^d}(1-1_{E_M}(x)) g(x)\,dx\to 0$ as
$M\to\infty$ by the Monotone Convergence Theorem. Therefore, there exists $M\in(0,\infty)$ such
that $G(E_{M})\leq \varepsilon/2$, so that
\begin{align*}
    G(S_j)& = G(S_j\cap E_M) + G(S_j\cap E_M^c)\leq \frac{\varepsilon}{2} + \int_{S_j\cap E_{M}^c}g(x)\,dx \\
    &\leq \frac{\varepsilon}{2} +M\lambda_d(S_j)
\end{align*}
and
\begin{equation}\label{eq:shell_lower}
    \lambda_d(S_j) \geq \frac{\varepsilon/2 - \delta_j}{M}.
\end{equation}
Since $S_j \subseteq \RR^d \setminus A_j$ and $A_j$ is closed, Lemma~\ref{lem:eikonal} gives
$|\nabla d_{A_j}(x)| = 1$ $\lambda_d$-almost everywhere on $S_j$. We apply the coarea
formula~\eqref{eq:coarea} to the map $x \mapsto d_{A_j}(x)$ over $S_j$ to obtain
\begin{align*}
\lambda_d(S_j) & = \int_{S_j} |\nabla d_{A_j}(x)| \, dx \\
&= c_{d-1}\int_0^{\delta_j} \mathcal{H}^{d-1}(\{x \in \RR^d : d_{A_j}(x) = t\}) \, dt.
\end{align*}
A non-negative integrable function must take a value at least equal
to its average on a set of positive measure; hence, there exists $t_j \in (0, \delta_j)$ such that
\begin{equation*}
\mathcal{H}^{d-1}(\{x \in \RR^d : d_{A_j}(x) = t_j\}) \geq \frac{1}{\delta_j} \int_0^{\delta_j} \mathcal{H}^{d-1}(\{x \in \RR^d : d_{A_j}(x) = t\}) \, dt.
\end{equation*}
Substituting the coarea identity and the lower bound \eqref{eq:shell_lower}, we obtain
\begin{equation*}
\mathcal{H}^{d-1}(\{x \in \RR^d : d_{A_j}(x) = t_j\}) \geq \frac{\lambda_d(S_j)}{c_{d-1}\delta_j} \geq \frac{\varepsilon/2 - \delta_j}{M c_{d-1}\delta_j}.
\end{equation*}
Since $\delta_j \to 0$ as $j \to \infty$, the proof is complete.
\end{proof}

Before moving to the next set of results, we highlight that, as the last display equation shows, our arguments not only deliver a diverging $(d-1)$-dimensional Hausdorff measure of the sequence of $(f_j,g,t_j)$-passing sets for some vanishing $t_j$, but they also provide an explicit lower-bound on the speed at which this divergence must happen. In particular, because the last inequality applies for any $\delta_j>d_{w}(F_j,G)$, we deduce that the divergence rate must be at least as fast as the inverse of the convergence rate of $F_j$ to $G$ in the L\'evy--Prokhorov metric. In the same spirit, one can obtain an analogous divergence result for $\mathcal{H}^{d-1}(\{x \in \RR^d : d_{A_j}(x) = t_j\})$ by relaxing the condition $d_{TV}(F_j,G)\geq \varepsilon$, which features a fixed $\varepsilon>0$, to allow for a sequence $\varepsilon_j$ going to 0 at a slower rate than $d_{w}(F_j,G)$. While we do not make use of these observations in the rest of this work, they are highlighted here as they may prove useful in contexts where convergence speed is of interest.

\subsection{Localization to a compact set}\label{subsec:compact}

Theorem~\ref{thm:oscillations_high_dim_hausdorff} reveals that a $(d-1)$-dimensional measure
diverges, but on an unbounded space this alone leaves open a somewhat uninteresting explanation: the sets
$A_j$ could simply be expanding in terms of their $\lambda_d$ content, so that their boundaries grow in the way the surface of an
inflating balloon does, without any oscillatory or pathological behavior at all. Borrowing from the one-dimensional case, what we wish to capture is
instead a set that folds and loops on itself within a fixed region, that is, growth in perimeter
that is genuinely due to complexity rather than size.

A first indication that the ``inflating balloon'' mechanism behind Theorem~\ref{thm:oscillations_high_dim_hausdorff} can be ruled out is detailed in the next proposition, which shows that the bulk of the larger-than-$\varepsilon$ total variation discrepancy between $F_j$ and $G$ arises within a fixed compact set.

\begin{proposition}\label{prop:compact_mass}
Under the assumptions of Theorem~\ref{thm:oscillations_high_dim_hausdorff}, for every
$\delta \in (0,\varepsilon)$ there exists a compact set $K_\delta \subset \RR^d$ such that
\begin{equation*}
F_j(A_j \cap K_\delta) - G(A_j \cap K_\delta) \geq \varepsilon - \delta
\qquad \text{for all } j \in \NN.
\end{equation*}
\end{proposition}

\begin{proof}
Since $d_w$ metrizes weak convergence, condition (ii) implies that $F_j$ converges weakly to $G$,
so that the family $\{F_j : j \in \NN\}$ is relatively compact and hence, by
Prokhorov's theorem, uniformly tight
\citep{billingsley1999}. Enlarging the resulting compact set so as to also
accommodate the tightness of the single measure $G$, we obtain a compact set
$K_\delta \subset \RR^d$ with
\begin{equation*}
F_j(K_\delta^c) < \delta/2\quad \text{for all } j \in \NN,
\qquad G(K_\delta^c) < \delta/2.
\end{equation*}
By TV-regularity and condition (i),
\begin{equation*}
\varepsilon \leq d_{TV}(F_j, G) = F_j(A_j) - G(A_j) = \int_{A_j} (f_j(x) - g(x))\,dx,
\end{equation*}
and splitting the domain of integration along $K_\delta$ and its complement gives
\begin{equation*}
\int_{A_j} (f_j - g)\,dx
= \int_{A_j \cap K_\delta} (f_j - g)\,dx + \int_{A_j \cap K_\delta^c} (f_j - g)\,dx.
\end{equation*}
Since $-g \leq g$ pointwise, the second term is bounded by
\begin{align*}
\int_{A_j \cap K_\delta^c} (f_j - g)\,dx
&\leq \int_{A_j \cap K_\delta^c} (f_j + g)\,dx
= F_j(A_j \cap K_\delta^c) + G(A_j \cap K_\delta^c) \\
&\leq F_j(K_\delta^c) + G(K_\delta^c) \leq \delta.
\end{align*}
Combining the last three displays yields
$\varepsilon \leq [F_j(A_j \cap K_\delta) - G(A_j \cap K_\delta)] + \delta$, as claimed.
\end{proof}

Proposition~\ref{prop:compact_mass} reveals that, however small $\delta > 0$ is chosen, there is a
compact set carrying all but $\delta$ of the total variation discrepancy, uniformly in $j$. While suggestive of the fact that the growth in perimeter documented by Theorem~\ref{thm:oscillations_high_dim_hausdorff} should not be entirely attributable to an increase in the area inside of it, Proposition~\ref{prop:compact_mass} on its own does not formally rule out that possibility. The next result closes this gap by showing that the
perimeter divergence persists even after intersection with a compact set $K$. It should be noticed that parts of the passing sets may still drift outside $K$, but the crux of the result is that the
divergence in perimeter must also occur inside $K$; as a consequence, the sets $A_j$ must become geometrically more
complicated within a fixed region of bounded Lebesgue measure.

\begin{theorem}\label{thm:oscillations_high_dim_hausdorff_localized}
Let $g, f_1, f_2, \dots$ be probability densities on $\RR^d$ with corresponding probability
distributions $G, F_1, F_2, \dots$. For each $j \in \NN$, assume that the pair $(F_j, G)$ is
TV-regular, and let $A_j := A(f_j,g)$. Suppose that, for some $\varepsilon>0$,
\begin{enumerate}
    \item[(i)] $d_{TV}(F_j, G) \geq \varepsilon$ for all $j \in \NN$,
    \item[(ii)] $d_w(F_j, G) \to 0$ as $j \to \infty$.
\end{enumerate}
Then, for any sequence $\delta_j > d_w(F_j, G)$ with $\lim_{j \to \infty} \delta_j = 0$, there
exists $t_j \in (0, \delta_j)$ and a compact set $K\subset \RR^d$ such that
\begin{equation*}
\lim_{j \to \infty} \mathcal{H}^{d-1}\bigl(P_{t_j}(f_j,g) \cap K\bigr) = \infty.
\end{equation*}
\end{theorem}

\begin{proof}
Preliminarily, because $G$ is a probability measure, for any $\delta \in (0, \varepsilon/2)$ there exists a compact set $K \subset \RR^d$
such that $G(K^c) < \delta$. By definition of the
total variation distance and the L\'evy--Prokhorov distance, the conditions
$d_{TV}(F_j, G) = F_j(A_j) - G(A_j) \geq \varepsilon$ and $\delta_j > d_w(F_j, G)$ imply
\begin{equation*}
\varepsilon + G(A_j) \leq F_j(A_j) \leq G(A_j^{\delta_j}) + \delta_j,
\end{equation*}
which yields
%\begin{equation*}
$G(A_j^{\delta_j} \setminus A_j) \geq \varepsilon - \delta_j$.
%\end{equation*}
Let $S_j :=\{x \in \RR^d : 0 < d_{A_j}(x) < \delta_j\} \equiv A_j^{\delta_j} \setminus A_j$.
Decomposing $S_j$ into its intersection with $K$ and $K^c$ gives
\begin{equation*}
G(S_j \cap K) = G(S_j) - G(S_j \setminus K) \geq G(S_j) - G(K^c) \geq \varepsilon - \delta_j - \delta.
\end{equation*}
Using the same reasoning as in the proof of Theorem~\ref{thm:oscillations_high_dim_hausdorff}, we
have
\begin{equation}\label{eq:localized_shell_lower}
\lambda_d(S_j \cap K) \geq \frac{\varepsilon/2 - \delta - \delta_j}{M}
\end{equation}
for some $M\in(0,\infty)$. Since
$S_j \cap K=(A_j^{\delta_j}\setminus A_j)\cap K\subseteq A_j^{\delta_j}\setminus A_j \subseteq \RR^d \setminus A_j$
and $A_j$ is closed, Lemma~\ref{lem:eikonal} implies $|\nabla d_{A_j}(x)| = 1$ almost everywhere
on $S_j \cap K$. Applying the coarea formula~\eqref{eq:coarea} to the map
$x \mapsto d_{A_j}(x)$ over $S_j \cap K$ yields
\begin{align*}
\lambda_d(S_j \cap K) & = \int_{S_j \cap K} |\nabla d_{A_j}(x)| \, dx \\
& = c_{d-1}\int_0^{\delta_j} \mathcal{H}^{d-1}(\{x \in K : d_{A_j}(x) = t\}) \, dt \\
& \equiv c_{d-1}\int_0^{\delta_j} \mathcal{H}^{d-1}(\{x \in \RR^d : d_{A_j}(x) = t\} \cap K) \, dt.
\end{align*}
By the mean value property for integrals, there exists a sequence
$t_j \in (0, \delta_j)$ such that
\begin{equation*}
\mathcal{H}^{d-1}(\{x \in \RR^d : d_{A_j}(x) = t_j\} \cap K) \geq \frac{\lambda_d(S_j \cap K)}{c_{d-1}\delta_j}.
\end{equation*}
Substituting \eqref{eq:localized_shell_lower} into the inequality gives
\begin{equation*}
\mathcal{H}^{d-1}(\{x \in \RR^d : d_{A_j}(x) = t_j\} \cap K) \geq \frac{\varepsilon/2 - \delta - \delta_j}{M c_{d-1}\delta_j}.
\end{equation*}
Since $\delta < \varepsilon/2$ is fixed, taking the limit as $j \to \infty$ with $\delta_j \to 0$
completes the proof.
\end{proof}

\subsection{Divergence of the number of connected components}\label{subsec:components}

A diverging $(d-1)$-dimensional measure can, in principle, be produced either by a growing number
of separate pieces or by a single piece of growing complexity. While the first mechanism cannot be assumed to hold in general (see Section~\ref{sec:example} for a counterexample), the next result isolates a simple sufficient condition for this to happen: if no single connected component is allowed to carry too much of the measure, then the
number of components must diverge. In particular, a constant bound (uniform in $j$) on such component-specific Hausdorff measures suffices.

\begin{proposition}\label{pro:components_diverge}
Under the same assumptions as in Theorem~\ref{thm:oscillations_high_dim_hausdorff}, $S_j$ as defined in the proof of Theorem~\ref{thm:oscillations_high_dim_hausdorff} is open
and hence it uniquely decomposes into at most countably many disjoint open connected components:
\begin{equation*}
S_j = \bigcup_{k=1}^{N_j} U_{j,k},
\end{equation*}
where $N_j \in \NN \cup \{\infty\}$. Therefore, letting $t_j \in (0, \delta_j)$ be the sequence
given by Theorem~\ref{thm:oscillations_high_dim_hausdorff}, if there exists a sequence $K_j > 0$
with $K_j = o(1/\delta_j)$ such that
\begin{equation*}
\mathcal{H}^{d-1}\bigl(P_{t_j}(f_j,g) \cap U_{j,k}\bigr) \leq K_j
\end{equation*}
for all $k\in\NN$, then
%\begin{equation*}
$\lim_{j \to \infty} N_j = \infty$.
%\end{equation*}
\end{proposition}

\begin{proof}
Either $N_j=\infty$ for all large enough $j\in\NN$, in which case the conclusion is trivial, or
$N_j$ is finite along a subsequence of indices $j$. In this case, by
Theorem~\ref{thm:oscillations_high_dim_hausdorff} and the fact that $H^{d-1}$ is a measure, we get
\begin{align*}
\sum_{k=1}^{N_j} \mathcal{H}^{d-1}(\{x \in U_{j,k} : d_{A_j}(x) = t_j\}) & = \mathcal{H}^{d-1}(\{x \in S_j : d_{A_j}(x) = t_j\}) \\
& \equiv \mathcal{H}^{d-1}(\{x \in \RR^d : d_{A_j}(x) = t_j\}) \\
& \geq \frac{\varepsilon/2 - \delta_j}{M c_{d-1}\delta_j}
\end{align*}
for some $M>0$. Using the uniform upper bound $K_j$ for each of the $N_j$ components, we have
\begin{equation*}
N_j K_j \geq \frac{\varepsilon/2 - \delta_j}{M c_{d-1}\delta_j} \implies N_j \geq \frac{\varepsilon/2 - \delta_j}{M c_{d-1} K_j \delta_j}.
\end{equation*}
Since $K_j = o(1/\delta_j)$, we have $K_j \delta_j \to 0$ as $j \to \infty$, completing the proof.
\end{proof}

We finish this section by emphasizing that the same conclusion holds for the passing sets themselves. Under the assumptions of
Proposition~\ref{pro:components_diverge}, let $M_j$ denote the number of connected components of
$P_{t_j}(f_j,g)$; then we claim that $M_j \to \infty$ as $j \to \infty$. To see this, notice first that the case in which $M_j$
is infinite for all large enough $j$ is obvious, so assume that there is a subsequence of indices
$j$ for which $M_j$ is finite. Each connected component $O_{j,k}$ of
$P_{t_j}(f_j,g)\subseteq S_j$ is included in one and only one connected component
$U_{j,k}$ of $S_j$, so that
\begin{equation*}
\mathcal H^{d-1}(O_{j,k})\leq \mathcal{H}^{d-1}(\{x \in U_{j,k} : d_{A_j}(x) = t_j\}) \leq K_j,
\end{equation*}
and therefore
\begin{equation*}
\frac{\varepsilon/2-\delta_j}{Mc_{d-1}\delta_j}\leq\mathcal{H}^{d-1}(\{x \in \RR^d : d_{A_j}(x) = t_j\}) \leq K_jM_j,
\end{equation*}
which, together with $K_j=o(1/\delta_j)$, implies $M_j\to\infty$.

\section{Illustrative examples}\label{sec:example}

To illustrate our theory, we first construct a sequence of probability measures $F_j$ on $\mathbb R^2$ satisfying the hypotheses of
Theorem~\ref{thm:oscillations_high_dim_hausdorff} and giving rise to a sequence of passing sets that consist of one single component at each step $j$. This first example serves two purposes: it illustrates Theorem~\ref{thm:oscillations_high_dim_hausdorff} and it shows that the additional hypothesis of
Proposition~\ref{pro:components_diverge} cannot be dispensed with, since the divergence of the
$(d-1)$-dimensional measure can be realized by a single connected component of growing geometric complexity.

Let $\Omega := [0,1]^2$ be the unit square in $\RR^2$. We define the limit probability density
$g : \RR^2 \to \RR$ as $g(x,y) := 1_\Omega(x,y)$, corresponding to the uniform probability measure
$G$ on $\Omega$. For any integer $j \geq 2$, we construct a connected set $A_j \subset \Omega$
composed of a horizontal ``bridge'' $B_j$ and a set of vertical ``teeth'' $T_j$:
\begin{align*}
    B_j &:= [0,1] \times \left[ \frac{1}{2} - \frac{1}{6j}, \frac{1}{2} + \frac{1}{6j} \right], \\
    T_j &:= \bigcup_{k=0}^{j-1} \left[ \frac{k}{j}, \frac{k}{j} + \frac{1}{3j} \right] \times [0, 1].
\end{align*}
Then let $A_j := B_j \cup T_j$. Because every vertical ``tooth'' intersects the central ``bridge,'' $A_j$
is a single connected component; see Figure~\ref{fig:Aj} for a visualization.

\begin{figure}
    \centering
    \includegraphics[width=0.95\linewidth]{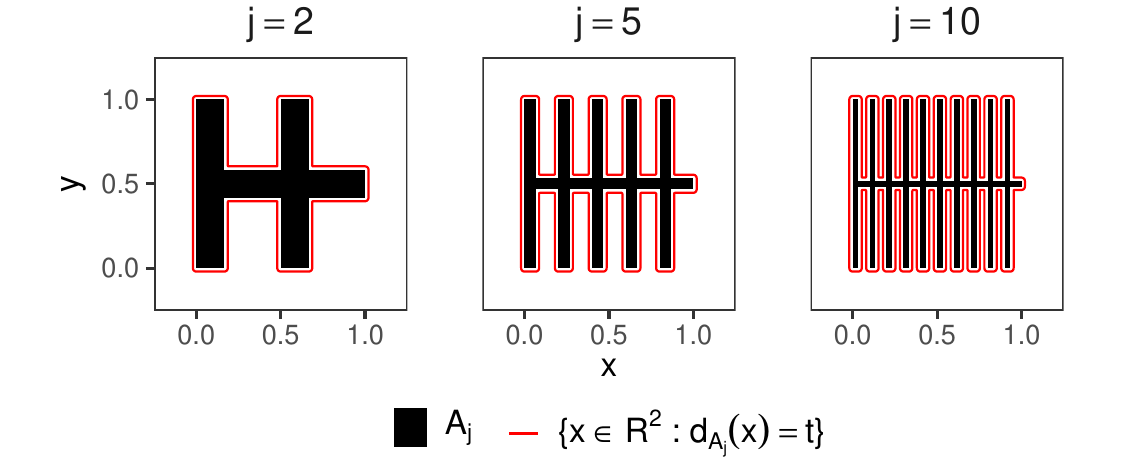}
    \caption{$A_j$ sets, shaded in black, for the ``bridge-and-teeth'' example. As $j$ increases, the length of the passing sets $P_t(f_j,g)$ (red line, plotted here for some small $t>0$) diverges to infinity, but thanks to the ``bridge'', the passing sets stay connected.}
    \label{fig:Aj}
\end{figure}

Let $p_j := \lambda_2(A_j)$ denote the Lebesgue measure of $A_j$, which is computed as follows:
\begin{equation*}
    p_j = \lambda_2(B_j) + \lambda_2(T_j) - \lambda_2(B_j \cap T_j) = \frac{1}{3j} + \frac{1}{3} - \frac{1}{3} \left( \frac{1}{3j} \right) = \frac{1}{3} + \frac{2}{9j}.
\end{equation*}
Note that for all $j \geq 2$, we have $p_j \leq \frac{1}{3} + \frac{1}{9} = \frac{4}{9} < \frac{1}{2}$.
Now define the sequence of probability densities $f_j : \RR^2 \to \RR$ as
\begin{equation*}
    f_j(x,y) := \left( 1 + 1_{A_j}(x,y) - \frac{p_j}{1-p_j}1_{\Omega \setminus A_j}(x,y) \right) 1_\Omega(x,y).
\end{equation*}
Since $p_j < 1/2$, it follows that $1 - p_j/(1-p_j) > 0$, ensuring $f_j(x,y) \geq 0$ for all
$(x,y) \in \RR^2$ and $j\geq 2$. Integrating over $\RR^2$ also yields $\int_{\RR^2} f_j \, dx \, dy = 1$, so
$f_j$ is a valid probability density on $\RR^2$ for all $j\geq 2$. Let $F_j$ denote the
probability measure associated with $f_j$.

For all $j\geq 2$, the pair $(F_j,G)$ is TV-regular: the set $A_j$ is closed and,
because it coincides with $\{x\in\RR^2 : f_j(x)>g(x)\}$, it achieves the total variation distance
between $F_j$ and $G$, so that $A_j = A(f_j,g)$ without loss of generality. We now verify that the sequence meets conditions
(i) and (ii) of Theorem~\ref{thm:oscillations_high_dim_hausdorff}.

\begin{proposition}\label{prop:no_tv}
The sequence of probability measures $F_j$ does not converge to $G$ in total variation distance. In particular, $d_{TV}(F_j,G)\geq 1/3$ for all $j\geq 2$.
\end{proposition}
\begin{proof}
By definition of the densities in the sequence, $f_j(x,y) = 2$ for $(x,y) \in A_j$, and $f_j(x,y) < 1$ for
$(x,y) \notin A_j$. Moreover, the supremum defining the total variation distance is achieved
by the set $A_j$:
\begin{equation*}
    d_{TV}(F_j, G) = F_j(A_j) - G(A_j).
\end{equation*}
To evaluate this difference, notice that $F_j(A_j) = 2 \lambda_2(A_j) = 2p_j$ and $G(A_j) = p_j$. Thus,
\begin{equation*}
    d_{TV}(F_j, G) = 2p_j - p_j = p_j = \frac{1}{3} + \frac{2}{9j}>\frac{1}{3}
\end{equation*}
for all $j\ge 2$, which completes the proof.
\end{proof}

\begin{proposition}\label{prop:weak}
The sequence of probability measures $F_j$ converges to $G$ in L\'evy--Prokhorov distance.
\end{proposition}
\begin{proof}
Since $d_{w}$ metrizes weak convergence, it suffices to show that for any bounded, continuous function $\phi \in C_b(\RR^2)$,
\begin{equation*}
    \lim_{j \to \infty} \int_{\RR^2} f_j(x,y) \phi(x,y) \, dx \, dy = \int_{\RR^2} g(x,y)\phi(x,y) \, dx \, dy.
\end{equation*}
Because both $f_j$ and $g$ evaluate to zero outside of $\Omega$, the integrals may be restricted
to $\Omega$ without loss of generality. First, we analyze the weak limit of the indicator function $1_{A_j}$. Write
\begin{equation*}
    \int_\Omega 1_{A_j} \phi(x,y) \, dx \, dy = \int_{T_j} \phi(x,y) \, dx \, dy + \int_{B_j \setminus T_j} \phi(x,y) \, dx \, dy.
\end{equation*}
Since $\phi$ is bounded, let $M := \sup_{(x,y)\in\RR^2} |\phi(x,y)|<\infty$, so that
\begin{equation*}
    \left| \int_{B_j \setminus T_j} \phi(x,y) \, dx \, dy \right| \leq M \lambda_2(B_j) = \frac{M}{3j} \to 0
\end{equation*}
as $j\to\infty$. Now define the marginal integral $\Phi(x) := \int_0^1 \phi(x,y) \, dy$, so that Fubini's theorem yields
\begin{equation*}
    \int_{T_j} \phi(x,y) \, dx \, dy = \sum_{k=0}^{j-1} \int_{k/j}^{k/j + 1/(3j)} \Phi(x) \, dx.
\end{equation*}
Since $\phi$ is continuous on the compact set $\Omega$, it is uniformly continuous
there, with some modulus of continuity $\omega_\phi$; that is, $|\phi(z) - \phi(z')| \leq
\omega_\phi(|z - z'|)$ for all $z, z' \in \Omega$, with $\omega_\phi(\delta) \to 0$ as
$\delta \to 0^+$. The marginal $\Phi$ inherits this modulus, since for $x, x' \in [0,1]$
\begin{equation*}
    |\Phi(x) - \Phi(x')|
    = \left| \int_0^1 \bigl( \phi(x,y) - \phi(x',y) \bigr) dy \right|
    \leq \int_0^1 \bigl| \phi(x,y) - \phi(x',y) \bigr| \, dy
    \leq \omega_\phi(|x - x'|),
\end{equation*}
the last step using the fact that $|(x,y) - (x',y)| = |x - x'|$. Thus $\Phi$ is uniformly continuous on $[0,1]$
with modulus $\omega_\phi$. On each sub-interval $[k/j,\, k/j + 1/(3j)]$, every point lies within
$1/(3j)$ of the left endpoint $k/j$, so that
\begin{equation*}
    \left| \int_{k/j}^{k/j + 1/(3j)} \Phi(x)\,dx - \frac{1}{3j}\,\Phi\!\left(\frac{k}{j}\right) \right|
    = \left| \int_{k/j}^{k/j + 1/(3j)} \left( \Phi(x) - \Phi\!\left(\frac{k}{j}\right) \right) dx \right|
    \leq \frac{1}{3j}\,\omega_\phi\!\left(\frac{1}{3j}\right).
\end{equation*}
Summing over the $j$ sub-intervals, the total error is bounded by
$\frac{1}{3}\,\omega_\phi(1/(3j)) = o(1)$, whence
\begin{equation*}
    \int_{T_j} \phi(x,y) \, dx \, dy = \frac{1}{3} \sum_{k=0}^{j-1}\frac{1}{j}  \Phi\left(\frac{k}{j}\right) + o(1).
\end{equation*}
The first term is one-third of a Riemann sum for $\Phi$ over $[0,1]$, so that
\begin{equation*}
    \lim_{j \to \infty} \int_{T_j} \phi(x,y) \, dx \, dy = \frac{1}{3} \int_0^1 \Phi(x) \, dx = \frac{1}{3} \int_\Omega \phi(x,y) \, dx \, dy.
\end{equation*}
Now, substituting $1_{\Omega \setminus A_j} = 1_\Omega - 1_{A_j}$ into the definition of $f_j$
yields
\begin{equation*}
    f_j = \left( \left( 1 - \frac{p_j}{1-p_j} \right) 1_\Omega + \left( 1 + \frac{p_j}{1-p_j} \right) 1_{A_j} \right) 1_\Omega = \frac{1-2p_j}{1-p_j} 1_\Omega + \frac{1}{1-p_j} 1_{A_j}.
\end{equation*}
As $j \to \infty$, $p_j \to \frac{1}{3}$, so
\begin{align*}
    \lim_{j \to \infty} \int_\Omega f_j(x,y) \phi(x,y) \, dx \, dy &= \frac{1 - 2(1/3)}{1 - 1/3} \int_\Omega \phi(x,y) \, dx \, dy \\
    &+ \frac{1}{1 - 1/3} \left( \frac{1}{3} \int_\Omega \phi(x,y) \, dx \, dy \right) \\
    & = \int_\Omega \phi(x,y) \, dx \, dy.
\end{align*}
Thus, $F_j$ converges weakly to $G$.
\end{proof}

The conditions of Theorems~\ref{thm:oscillations_high_dim_hausdorff} and \ref{thm:oscillations_high_dim_hausdorff_localized} are met, which implies that
the length of the passing set $P_{t_j}(f_j,g)$, for some $t_j\to0$, diverges to infinity (and it clearly does so in a compact set).
However, the conclusions of Proposition~\ref{pro:components_diverge} do not follow, as the set $A_j$
consists of only a single connected component and so do the sets
$\{x\in\RR^2 : 0<d_{A_j}(x)<\delta\}$ and $P_t(f_j,g)$ for any $\delta,t>0$;
see again Figure~\ref{fig:Aj} for a visualization (as a red line) of
$P_t(f_j,g)$ when $t$ is small (for large $t$, it is also clear that
$P_t(f_j,g)$ has a single component). This example is particularly helpful to understand why the one-dimensional intuition of an increasing number of oscillations, in the form of a passing set fragmenting in a growing number of components, fails in general when $d\geq 2$: in more than one dimension, there is an infinite number of directions along which oscillatory behavior may happen, and because the latter may happen along some directions but not others, the resulting passing sets may remain connected along the ``non-oscillatory directions.'' This is precisely what emerges from Figure~\ref{fig:Aj}, where an increasingly oscillatory marginal density for the $x$-axis variable is paired by a nearly uniform marginal density for the $y$-axis coordinate, where a thinning ``bridge'' connects the horizontal ``teeth'' or oscillations.

The previous example has shown that, for the conclusions of Theorems~\ref{thm:oscillations_high_dim_hausdorff} and \ref{thm:oscillations_high_dim_hausdorff_localized} to hold in $d\geq 2$, the passing set need not fragment into an increasing number of components. We now show that this fragmentation, while not necessary, may nevertheless occur in practice, producing a phenomenon that is more in line with the one-dimensional intuition of an increasing number of oscillations. To that end, first note that the sequence of densities we just analyzed, once it is appropriately modified, can be used to produce an example of a sequence for which the number of connected components of $P_{t_j}(f_j,g)$, for some $t_j\to 0$, goes to infinity. In particular, redefining $A_j:=T_j$, that is, removing the ``bridge'' from the previous construction, the convergence analysis of the resulting sequence of probability measures remains virtually unchanged, while choosing $t_j=1/(6j)$,  $P_{t_j}(f_j,g)$ is easily seen to consist of $j$ components (each ``tooth'' is at distance $2/(3j)$ from the nearest distinct ``tooth''). Hence, the number of components of $P_{t_j}(f_j,g)$ diverges to infinity, in accordance with Proposition~\ref{pro:components_diverge} and the discussion following it.

We close this section with a related example, which is closer
in spirit to the trigonometric sequence
discussed in Sections~\ref{sec:introduction} and \ref{sec:from1d} (recall Figure~\ref{fig:cosine_1d}) and in which the passing sets also fragment into a
diverging number of components. Let $\Omega:=[0,1]^2$, $g:=1_\Omega$, and for $j\in\NN$ define
\begin{equation*}
f_j(x,y) := \bigl(1+\cos(2\pi j x)\bigr)\bigl(1+\cos(2\pi j y)\bigr)\,1_\Omega(x,y).
\end{equation*}
Since $\int_0^1 (1+\cos(2\pi j x))\,dx = 1$ for every $j$, Fubini's theorem ensures $\int_{\Omega} f_j(x,y)\,dx\,dy=1$, so that $f_j$ is a valid probability density; the associated measure is
denoted $F_j$. This is the most intuitive two-dimensional analogue of the
one-dimensional oscillating sequence $x\mapsto 1+\cos(2\pi j x)$, obtained by taking
the product of two such factors. Clearly, the construction can be generalized to $d>2$ without difficulty.

That $F_j\to G$ weakly follows along the same lines as in the proof of Proposition~\ref{prop:weak}: for $\phi\in C_b(\RR^2)$, the Riemann--Lebesgue lemma
applied in each variable gives $\int_\Omega f_j(x,y)\,\phi(x,y)\,dx\,dy \to \int_\Omega \phi(x,y)\,dx\,dy$, since
every term involving $\cos(2\pi j x)$ or $\cos(2\pi j y)$ vanishes in the
limit. On the other hand $F_j$ does not converge to $G$ in total variation: writing
$C_j := \{(x,y)\in\Omega : f_j(x,y)>1\}$, one has
$d_{TV}(F_j,G) = \int_{C_j}(f_j(x,y)-1)\,dx\,dy$, and a direct computation shows that this
quantity is bounded away from zero uniformly in $j$, as the average of
$(f_j-1)_+$ over each period cell does not depend on $j$. The pair $(F_j,G)$ is also
TV-regular, with $A_j := A(f_j,g)$ taken to be $\overline{C_j}$, which achieves the total variation distance because
$\overline{C_j}\setminus C_j$ is $\lambda_2$-null. Therefore Theorems \ref{thm:oscillations_high_dim_hausdorff} and \ref{thm:oscillations_high_dim_hausdorff_localized} apply, so that there exists $t_j\to 0$ such that $\mathcal H^1(P_{t_j}(f_j,g))\to \infty$ as $j\to \infty$. 

The set $C_j$ consists of one connected component within each of the $(j+1)^2$ 
squares of side $1/(j+1)$ making up $\Omega$, so it has $(j+1)^2$ connected
components; see Figure~\ref{fig:cosine} for a visual illustration. Consequently, for every sufficiently small $t>0$ the passing set
$P_t(f_j,g)$ likewise splits into $(j+1)^2$ components, one enclosing
each component of $A_j$, and both its total $\mathcal H^{1}$ measure and its number of
components diverge as $j\to\infty$; see again Figure~\ref{fig:cosine}. This is in line with the fact that the hypotheses
of Proposition~\ref{pro:components_diverge} are met, since the perimeter of each component of small enlargements of $A_j$ is trivially bounded by 4 (the perimeter of $\Omega$). This in turn implies that, unlike in the ``teeth-and-bridge'' example, the divergence of $\mathcal H^{1}(P_{t_j}(f_j,g))$ is
accompanied by increasing fragmentation of the passing set, recovering the one-dimensional picture
of a diverging number of oscillations of $f_j$ around $g$.

\begin{figure}[t]
    \centering
    \includegraphics[width=0.95\linewidth]{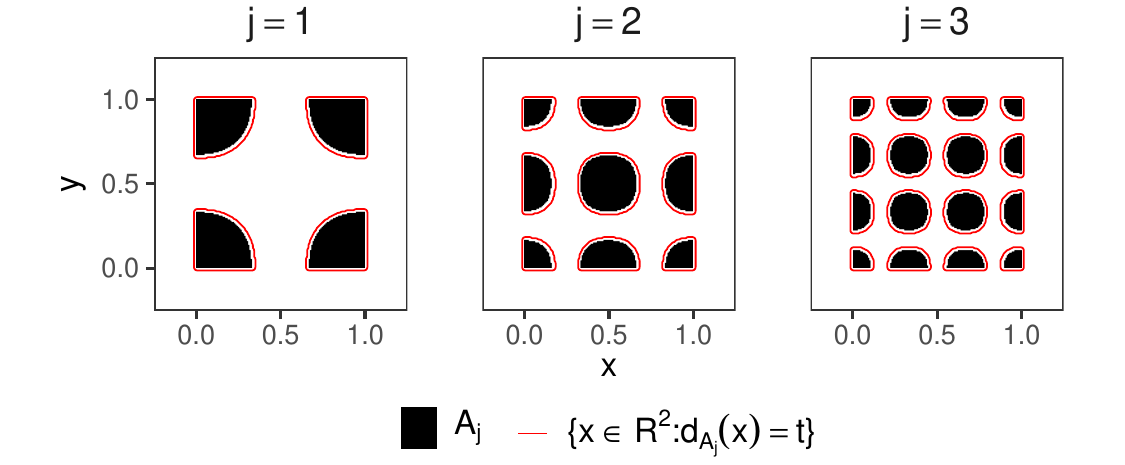}
    \caption{The sets $A_j=\overline{\{f_j>g\}}$ (black) for the
    densities $f_j(x,y)=(1+\cos(2\pi j x))(1+\cos(2\pi j y))$ and $g=1$ on $[0,1]^2$, as $j$ increases. $A_j$ breaks into $(j+1)^2$ disjoint
    components; the passing sets $P_t(f_j,g)$ (red, for a small $t>0$) enclose each component
    separately, so that both their total length and their number of connected
    components diverge, in contrast with the ``bridge-and-teeth'' example of
    Figure~\ref{fig:Aj}.}
    \label{fig:cosine}
\end{figure}

\section{Conclusion}\label{sec:conclusion}

We have shown that, for probability measures on $\RR^d$ admitting densities with respect to the
Lebesgue measure, weak convergence in the absence of total variation convergence forces the
$(d-1)$-dimensional Hausdorff measure of the density passing sets to diverge, at a rate inversely
proportional to the L\'evy--Prokhorov convergence rate, and that the divergence takes place within
a compact set. The one-dimensional notion of a diverging number of oscillations is recovered as
the case $d=1$, but the example of Section~\ref{sec:example} shows that it does not in general
survive the passage to higher dimensions in its original form: the growing geometric complexity of the sets $A(f_j,g)$ need not manifest itself as a fragmentation into distinct components, and
is more generally captured by their size in the sense of $\mathcal H^{d-1}$.

Two directions for future investigation are
worth pointing out. First, whether any kind of converse result holds remains an open question: it would be of theoretical and practical interest to determine to what
extent the divergence of the passing sets' measure is not merely a consequence but a
characterization of the failure of total variation convergence under weak convergence. Second, the results obtained in this article may find interesting applications in Bayesian asymptotics, which provided the
original motivation for this work.

\bibliographystyle{plainnat}
\bibliography{arxiv}

\end{document}